\documentclass{amsart}
\usepackage{amsthm,amssymb,wasysym}
\newtheorem{theorem}{Theorem}

\newtheorem{lemma}[theorem]{Lemma}

\newtheorem{corollary}{Corollary}

\newtheorem{remark}{Remark}

\begin{document}
\title[Polynomials ]{Some asymptotics for the Bessel functions with an explicit error term}

\author[I. Krasikov]{Ilia Krasikov }

\address{   Department of Mathematical Sciences,
            Brunel University,
            Uxbridge UB8 3PH United Kingdom}
\email{mastiik@brunel.ac.uk}


\date{4.07.11}


\begin{abstract}
We show how one can obtain an asymptotic expression for some special functions satisfying a second order differential equation with a very explicit error term starting from appropriate upper bounds.
   We will work out the details for the Bessel function $J_\nu (x)$ and the Airy function $Ai(x)$ and find a sharp approximation for their zeros.
   We also answer the question raised by Olenko by showing that
   $$c_1 \left| \nu^2-\frac{1}{4}\,\right| < \sup_{x \ge 0} x^{3/2}\left|J_\nu(x)-\sqrt{\frac{2}{\pi x}} \, \cos \left( x-\frac{\pi \nu}{2}-\frac{\pi}{4}\, \right)\right| <c_2 \left|\nu^2-\frac{1}{4}\,\right| ,
$$
$ \nu \ge -\frac{1}{2} \, ,$ for some explicit numerical constants $c_1$ and $c_2.$

\hspace{1ex}

\noindent
\end{abstract}

\maketitle
\section{Introduction}
All basic formulas and asymptotic expressions for special functions we use  without references can be found in \cite{olv}. To write down error terms in a compact form we will use $\theta, \theta_1,\theta_2,...,$ to denote quantities with the absolute  value not exceeding one.

In most of the cases error terms of asymptotics of special functions are either not known or, at best, valid for a rather restricted rang of parameters.
The following is a typical example of that kind (see e.g. \cite[Ch. 10]{olv}).

The Bessel function $J_\nu (x)$ is defined by the series
\begin{equation}
\label{seriesbes}
J_\nu (x) =\left(\frac{x}{2} \right)^\nu \sum_{j=0}^\infty (-1)^j \frac{(x^2/4)^j}{j! \, \Gamma(j+\nu+1)}
\end{equation}
and is a solution of the following ODE:
\begin{equation}
\label{difeqbes}
x^2 J''_\nu (x)+x J'_\nu (x)+(x^2-\nu^2) J_\nu (x)=0.
\end{equation}
\begin{theorem}
\label{asybes}
Suppose that $\nu \ge 0,$ $x >0,$  $\omega_\nu =\frac{\pi \nu}{2}+\frac{\pi}{4} \, ,$ and let
$$\ell_1 \ge \max(\frac{\nu}{2}-\frac{1}{4}\, ,1), \; \;\; \ell_2 \ge \max(\frac{\nu}{2}-\frac{3}{4}\, ,1),$$
then
\begin{equation}
\label{olivapprox}
\sqrt{\frac{\pi x}{2}} \;J_{\nu}(x) =
\cos{(x- \omega_\nu)} \left( \sum_{i=0}^{\ell_1-1}
  \frac{a_{2i}(\nu )}{x^{2i} } + \theta_1^2 \, \frac{a_{2\ell_1}(\nu )}{x^{2\ell_1} }\right) -
  \end{equation}
  $$
\sin{(x-\omega_\nu )} \left(\sum_{i=0}^{\ell_2-1}
 \frac{a_{2i+1}(\nu )}{x^{2i+1} } + \theta_2^2 \, \frac{a_{2\ell_2+1}(\nu )}{x^{2\ell_2+1} }  \right), $$
where
$$a_i (\nu )= \frac{(\frac{1}{2}-\nu)_i (\frac{1}{2}+\nu)_i}{2^i i!}
. $$
\end{theorem}

The assumption $\nu \ge 0$ is not really restrictive and can be surmount by, say, applying the three term recurrence for $J_\nu.$ However if $\nu$ is large or depends on $x$ estimating the error term in (\ref{olivapprox}) seems at least as difficult as the original task.

In this paper we show that there is a simple way to circumvent this problem and to find an explicit expression for error terms which is also uniform in the parameters, provided one has an a priory upper bound on the absolute value of the considered function.
In turn, in many cases such a bound may be obtained by using so-called Sonin's function. For Bessel and Airy functions, as well as for Hermite polynomials (see \cite{flk}), this can be done in a quite rutin way. For Jacobi and Laguerre polynomials it is a much more involved problem and the result is known only for oscillatory and transition regions \cite{kras07},\cite{kras08},\cite{kras10}. It is worth noticing that despite the fact that it is rather a technical problem and we do have appropriate tools to tackle it (see e.g. Lemmas \ref{monoton} and \ref{leftmax} below), one still needs a good deal of calculations to extend the bounds to monotonicity region. Thus, although the underlying idea of the method we use here is quit simple and can be applied to many special functions satisfying a second order differential equation, it is not utterly straightforward to work out the details. In this paper we will consider the Bessel function $J_\nu(x)$ as an important example to illustrate this approach. We provide asymptotic expressions with an explicit error term for oscillatory and transition regions and also give some new estimates in the monotonicity region. As a corollary we derive a surprisingly accurate approximation for the Airy function $Ai(-x), \; x >0,$ and obtain an approximation with an explicit error term for its positive zeros, which, due to  known inequalities, yields an approximation of the zeros of the Bessel function.

It is worth noticing that the obtained estimates are, in a sense, best possible. In particular we will answer a question raised by Olenko \cite{olenko} by showing that for $\nu \ge -\frac{1}{2} \, ,$
\begin{equation}
\label{olquest}
c_1 \left|\nu^2-\frac{1}{4}\,\right| < \sup_{x \ge 0} x^{3/2}\left|J_\nu(x)-\sqrt{\frac{2}{\pi x}} \, \cos \left( x-\frac{\pi \nu}{2}-\frac{\pi}{4}\, \right)\right| <c_2 \left|\nu^2-\frac{1}{4}\,\right|,
\end{equation}
for some explicit numerical constants $c_1$ and $c_2.$

The paper is organized as follows. In the next section we describe an idea of the method. In section \ref{secup} we establish some upper bounds on Bessel and Airy functions we need in the sequel. Our main tool here will be Sonin's function. In section \ref{secosc}
we consider the error term of the standard asymptotic
$$ J_\nu(x) \sim \sqrt{\frac{2}{\pi x}} \, \cos \left( x-\frac{\pi \nu}{2}-\frac{\pi}{4} \right),$$
and prove (\ref{olquest}), thus answering Olenko's question. The error term for asymptotic in the transition region will be derived in section \ref{sectran}.
Finally, in section \ref{secanother}, using the approach of section \ref{secprel}, we establish a different sharper approximation for
Bessel and Airy functions and their zeros.

\section{Preliminaries }
\label{secprel}
The following approach was shortly described in \cite{kras10}.
We want to find an approximation of a solution of the differential equation
\begin{equation}
\label{normformeq}
f''+b^2(x)f(x)=0,
\end{equation}
in terms of some standard function $F(x),$ which also satisfies a second order ODE
$$\mathcal{D}_1(F)=p_2 (x)F''+p_1(x)F'+p_0(x)F=0.$$
In fact, in what follows we choose $F$ to be just $\cos \phi(x)$ with an appropriate function $\phi.$

A possible approach to this problem is to seek for a multiplier function $z(x)$ such that
the differential operator
$$\mathcal{D}_2(g)=q_2 (x)g''+q_1(x)g'+q_0(x)g$$
for $g=g(x)=z(x)f(x)$ is in some sense close to $\mathcal{D}_1.$
For example, in the WKB-type approximation one chooses $g(x)=\sqrt{b(x)} \, f(x),$ yielding
\begin{equation}
\label{epsileq}
g''-\frac{b'}{b} \, g' +b^2 g \left(1 +\epsilon(x) \right)=0, \; \; \; \epsilon(x)=\frac{3b'^2-2b b''}{4b^4} \, .
\end{equation}
If $\epsilon$ is small we can expect that $g(x)$ is close to the solution of the equation
$$g_0''-\frac{b'}{b} \, g_0' +b^2 g_0=0,$$
which is just $g_0=M \cos \mathcal{B}(x),$
where $\mathcal{B}(x)=\int b(x) d x.$

Assume now that we have an a priori bound $|g(x)| \le C.$ Then we can readily estimate the error term $|g-g_0|$ by solving (\ref{epsileq}) as an inhomogeneous equation,
\begin{equation}
\label{eqwkb}
g(x)=g_0-\int^x \epsilon (t)b(t)\sin \left(\mathcal{B}(x)-\mathcal{B}(t) \right) g(t) dt ,
\end{equation}
thus obtaining
$$g=g_0+ \theta  C \int^x \epsilon (t)b(t) dt . $$

To derive an upper bound on $|g(x)|$ we consider Sonin's function
$$S(x)=S(g;x)=g^2+\frac{g'^2}{b^2(1+\epsilon)}  =g^2+ \frac{4b^2}{4b^4-2b b''+3b'^2} \, g'^2 ,$$
then
$$S'(x)= \frac{8b(6b'^3-6b b' b''+b^2 b''')}{(4b^4+3b'^2-2b b'')^2} g'^2 .$$
Thus, if $4b^4-2b b''+3b'^2 >0$ and $6b'^3-6b b' b''+b^2 b'''>0$ then $S' >0,$ and we obtain
$g^2(x) <S(\infty ).$

Moreover, one can also get an upper bound on $S$ in the following way:
$$S-\frac{b(4b^4+3b'^2-2b b'')}{2(6b'^3-6b b' b''+b'^2 b''')} \, S'=g^2 \ge 0,$$
that is
$$S'/S \le \frac{2(6b'^3-6b b' b''+b'^2 b''')}{b(4b^4+3b'^2-2b b'')} =\frac{d}{dx} \, ln \frac{b^4}{4b^4+3b'^2-2b b''} \, ,$$
provided the last expression is nonnegative.
Integrating, we find
$$\frac{S(y)}{S(x)} \le \frac{1+ \frac{3b'^2(x)}{4b^4(x)}-\frac{b''(x)}{2b^3(x)}}{1+ \frac{3b'^2(y)}{4b^4(y)}-\frac{b''(y)}{2b^3(y)}} =\frac{1+\epsilon (x)}{1+\epsilon (y)}\, .$$
This shows that the envelop of $g^2(x)$ given by $S(x)$ is almost a constant as far as $\epsilon (x)=o(1).$

In practically important examples the situation is somewhat more subtle as the coefficient $b(x)$ may vanish. For instance for the Bessel function $b(x)=x^{-1} \sqrt{x^2-\nu^2+1/4}\, ,$ and Sonin's function does not provide any information for the monotonicity region $0 <x < \sqrt{\nu^2-1/4}\, .$ Thus, one needs some supplementary estimates to extend the bounds on $|g(x)|$ to this interval. Let us notice that although the behaviour of the solutions of (\ref{normformeq}) looks less complicate in the monotonicity region, it, probably, allows only a piecewise approximation in reasonably simple elementary functions.

Another rather technical problem is how to find the constants of integration in $g_0.$ Here one either has to know, at least approximately, the value of $g(x)$ at some points, e.g. at infinity, or to be able to match asymptotics in the oscillatory and transition regions.

It is worth noticing that an asymptotic with an explicit error term in the transition region, that is around a zero of $b(x),$ can be obtained quite directly, provided we know a bound on $|g(x)|.$ Indeed, let $b(\alpha)=0,$ and let
$d=\left. \frac{d}{dx}\, b^2(x) \right|_{x=\alpha} ,$ then in a vicinity of $\alpha$ we can write
$$b^2(\alpha+d^{-1/3} t)=d^{2/3} t +\delta(t),  $$
where $\delta(t)$ is small.
The function $y(t)=f(\alpha+d^{-1/3} t)$ satisfies an Airy type ODE
$$y''(t)+t y(t) =-\delta(t) d^{-2/3} y(t).$$
Solving it as an inhomogeneous equation one gets an explicit error term as above, yet facing again the problem of fixing integration constants.

\section{Upper bounds}
\label{secup}
 For the Bessel function it will be convenient to introduce the parameter $\mu = |\nu^2- \frac{1}{4} \, |.$  We also use the above notation $\omega_\nu =\frac{\pi \nu}{2}+\frac{\pi}{4} \, .$

In this section we establish some upper bounds we need in the sequel.
A simplest inequality of this type \cite{watson} states that for $x$ real
\begin{equation}
\label{watsineq}
J_\nu (x) \le \frac{|x|^\nu }{2^\nu \Gamma (\nu+1)} \, .
\end{equation}
For our purposes we need much more accurate estimates.
To bound the Bessel function in the monotonicity region we will apply the following inequality given in \cite{kras06} . We sketch a proof for self-completeness.
\begin{lemma}
\label{monoton}
Let $\mathcal{J}_\nu(x)=x^{-\nu} J_\nu (x), \; \; \nu \ge - \frac{1}{2} \, ,$ then for $0 < x \le \nu+\frac{1}{2} \, ,$
\begin{equation}
\label{eqmonoton}
\frac{\mathcal{J}_\nu'(x)}{\mathcal{J}_\nu(x)} \ge \frac{\sqrt{(2\nu+1)^2-4x^2}-2\nu-1}{2x} \ge  - \frac{2x}{2\nu+1}\, .
\end{equation}
\end{lemma}
\begin{proof}
 The function $\mathcal{J}_\nu(x)$ is an entire function with only real zeros satisfying the Laguerre inequality $\mathcal{J}_\nu'^2-\mathcal{J}_\nu \mathcal{J}_\nu'' >0.$  Substituting here $\mathcal{J}_\nu''$ from the differential equation
$$x \mathcal{J}_\nu''+(2\nu+1)\mathcal{J}_\nu'+x \mathcal{J}_\nu=0, $$ and dividing by $\mathcal{J}_\nu^2/x$
we obtain
$$ x t^2(x)+(2\nu+1)t(x)+x \ge 0,$$
where $t(x)=\mathcal{J}_\nu'/\mathcal{J}_\nu.$
Hence,
$$t(x) \notin \left( - \frac{\sqrt{(2\nu+1)^2-4x^2}+2\nu+1}{2x} \, ,\frac{2x}{\sqrt{(2\nu+1)^2-4x^2}+2\nu+1}\right) .$$
Since $t(x) \rightarrow 0$ when $x \rightarrow 0^+,$ whereas $- \frac{\sqrt{(2\nu+1)^2-4x^2}+2\nu+1}{2x}\rightarrow - \infty,$
we conclude that
$$t(x) \ge - \frac{2x}{\sqrt{(2\nu+1)^2-4x^2}+2\nu+1} \ge - \frac{2x}{2\nu+1}\, .$$
\end{proof}
We need the following very accurate estimate giving the value of  $J_\nu(\nu):$
\begin{equation}
\label{besnunu}
J_\nu(\nu)=\frac{2^{1/3}}{3^{2/3} \Gamma (\frac{2}{3} \,) (\nu+\theta^2 \alpha)^{1/3}} \, , \; \; \; \nu >0,
\end{equation}
where $\alpha=0.09434980...$, \cite{elber1}.

 The following theorem improves the inequality
$$J_\nu (t \nu)< J_\nu (\nu) t^\nu e^{(1-t)\nu} \,, \; \; \;  \nu>0, \; \; 0<t<1. $$
given in \cite{paris}. For large $\nu$ it is also stronger than the classical inequality
$$J_\nu (x ) < \frac{x^\nu}{2^\nu \Gamma (\nu+1)} \, e^{-x^2/4(\nu+1)} ,$$
\cite[p. 16]{watson}, provided $t >\sqrt{\ln 16-2} \approx 0.88.$

\begin{theorem}
\label{monotonoz}
For $x=t \nu, \; \; 0 <t \le 1$ and $\nu >0,$
\begin{equation}
\label{ozmonj}
J_\nu (t \nu) \le  J_\nu (\nu) t^\nu \exp{\left(\frac{\nu^2(1 - t^2)}{2\nu+1} \right)} < \frac{2^{1/3} x^\nu}{3^{2/3} \Gamma (\frac{2}{3}) \nu^{\nu+1/3}}\,
\exp{\left(\frac{\nu^2 - x^2}{2\nu+1} \right)}.
\end{equation}
\end{theorem}
\begin{proof}
By the previous lemma we have
$$\ln \frac{\mathcal{J}_\nu(\nu)}{\mathcal{J}_\nu(x)} \ge - \int_x^\nu  \frac{2z}{2\nu+1} \, dz = - \frac{\nu^2 - \; x^2}{2\nu+1},$$
hence
$$\mathcal{J}_\nu(t \nu) \le  \mathcal{J}_\nu(\nu)\exp{\left(\frac{\nu^2(1 - t^2)}{2\nu+1} \right)},$$
which together with (\ref{besnunu})
 yields the required result.
\end{proof}
\begin{remark}
The function $\mathcal{J}_\nu(x)=x^{-\nu} J_\nu (x)$ of Lemma \ref{monoton} belongs to so-called P\'{o}lya-Laguerre class and satisfies the infinite
series of inequalities:
\begin{equation}
\label{pjen}
L_{m}(\mathcal{J}_\nu)=\sum_{j=0}^{2m} (-1)^{m+j} \frac{{2m \choose j}}{(2m)!} \mathcal{J}_\nu^{(j)} \mathcal{J}_\nu^{(2m-j)} \ge 0,
\end{equation}
where $L_1$ is the usual Laguerre inequality $\mathcal{J}_\nu'^2-\mathcal{J}_\nu \mathcal{J}_\nu'' \ge 0,$ (see e.g. \cite{patrick1}, \cite{patrick2}).
Using $L_m(\mathcal{J}_\nu)>0$ for $m>1$ leads to much more precise yet more complicated bounds on $\mathcal{J}_\nu'/\mathcal{J}_\nu$ and consequently on $J_\nu.$ Alternatively,
one can use the inequality $L_1 (\mathcal{J}_\nu+\lambda \mathcal{J}'_\nu ) \ge 0,$
$\lambda \in \mathbb{R},$ then optimizing in $\lambda.$ It is worth noticing that both methods give an inequality similar to (\ref{eqmonoton}) but in the opposite direction. Thus, one can use the known value of $\mathcal{J}_\nu (0)$ instead of $\mathcal{J}_\nu (\nu).$
\end{remark}

Our main tool for bounding solutions of the second order differential equations will be  Sonin's function.
In particular, it was used by Szeg\"{o} to prove that for $|\nu| \le \frac{1}{2} \, ,$
\begin{equation}
\label{besszegj}
|J_\nu(x)|  \le \sqrt{\frac{2}{\pi x}}\,  ,
 \end{equation}
 Although he did not state this explicitly, his proof of Theorem 7.31.2. immediately implies
 \begin{equation}
 \label{besszegy}
 |Y_\nu(x)|  \le \sqrt{\frac{2}{\pi x}}\,.
 \end{equation}
 His arguments go as follows:
 let $y$ be a solution of the Bessel differential equation
 $$ x^2 y''+x y'+(x^2-\nu^2)y=0,$$
 then for $|\nu| \le \frac{1}{2}$ Sonin's function
 \begin{equation}
 \label{szegsonin}
 S(x)=x y^2+\frac{x^2}{x^2 - \mu} \, \left(\frac{d}{dx}\, \sqrt{x }\,y \right)^2
 \end{equation}
is increasing and inequalities (\ref{besszegj}) and (\ref{besszegy}) follow by calculating $S(\infty)$ from known asymptotics of $J_\nu$ and $Y_\nu,$
whereas for $\nu >\frac{1}{2}$ it is decreasing for $x> \sqrt{\mu },$ and does not lead, at least directly, to any explicit inequality.

It turns out that for $\nu > 1/2$ it is more natural to deal with the function
\begin{equation}
\label{defh}
\mathcal{H}_\nu(x)=|x^2-\mu |^{1/4} J_\nu (x),
\end{equation}
 rather than $ \sqrt{x} \, J_\nu (x).$
Here we will refine an inequality for the Bessel function obtained in \cite{flk}.
First we need a bound on location of the leftmost maximum of $\mathcal{H}_\nu(x).$
\begin{lemma}
\label{leftmax}
The first positive maximum of $\mathcal{H}_\nu(x), \; \; \nu \ge \frac{5}{3} \, ,$ is attained at a point $\xi$ satisfying
$$\xi > \nu \sqrt{1-(2\nu)^{-2/3}}\,$$
\end{lemma}
\begin{proof}
Since obviously $0<\xi < \mu,$ we can restrict ourselves to the interval $(0,\mu)$ and write down
$$\mathcal{H}_\nu (x)=x^\nu (\mu-x^2)^{1/4} \mathcal{J}_\nu(x),$$
where as above $\mathcal{J}_\nu(x)=x^{-\nu} J_\nu (x).$
Then
$$0=\mathcal{H}'_\nu(\xi)= \frac{\xi^{\nu-1}}{2(\mu-\xi^2)^{3/4}}
\left(2(\mu-\xi^2)(\xi t(\xi)+\nu)-\xi^2 \right)\mathcal{J}_\nu(\xi),$$
where $t(x)=\mathcal{J}_\nu'(x)/\mathcal{J}_\nu(x).$ Hence
$$t(\xi)= - \frac{(2\nu+1)(2\nu^2-\nu-2\xi^2)}{(4\nu^2-1-4\xi^2)\xi} \, , $$
and comparing this with (\ref{ozmonj}) we obtain the inequality
$$\frac{(2\nu+1)(2\nu^2-\nu-2\xi^2)}{4\nu^2-1-4\xi^2} \le
\frac{2\nu+1-\sqrt{(2\nu+1)^2-4\xi^2}}{2} \, .$$
Simplifying we get
$$16\xi^6-4(2\nu+1)(6\nu-1)\xi^4+(2\nu-1)(2\nu+1)^2(6\nu+1)\xi^2+\nu(\nu+1)(4\nu^2-1)^2 :=p(\xi) \ge 0.$$
Observe that for $\nu \ge 5/3$ this polynomial has the only positive zero $\xi_0,$
Indeed, the discriminant of $p(\xi)$ in $\xi,$
 up to an irrelevant numerical factor, is
 $$ \nu (\nu+1)(2\nu-1)^6 (2\nu+1)^{10} (108 \nu^2-172\nu-5)^2.$$
 Thus the number of positive zeros does not change for $108 \nu^2-172\nu-5 >0,$
 in particular for $\nu >5/3 \, .$
 For $\nu=5/2$ we obtain the following test equation
 $\xi^6-21\xi^4+144\xi^2-315=0,$
 with the only positive zero $\xi \approx 2.14.$
 Finally,
 $$p(\nu )=\nu (4{\nu}^3-2\nu-1)>0,$$
 and using the substitution $n=r^3/2,$ we find
 $$p(\nu \sqrt{1-(2\nu)^{-2/3}}\; )= p(\sqrt{r^6-r^4} /2)=-r^3 (r^2-1)^2(2r^4+4r^2-r+2) <0,$$
 hence
 $$\xi >\nu \sqrt{1-(2\nu)^{-2/3}}\,. $$
\end{proof}

\begin{theorem}
\label{bessel}
Let $\nu > \frac{1}{2} \, ,$ then for $x \ge 0 ,$
\begin{equation}
\label{eqbess}
| x^2-\mu |^{1/4} |J_\nu (x)| < \sqrt{\frac{2}{\pi}}\, ,
\end{equation}
and the constant $\sqrt{\frac{2}{\pi}}\, $ is best possible.
\end{theorem}
\begin{proof} For $x = \sqrt{\mu} $ the result is trivial. Otherwise
we shall consider three cases.
\\
{\it Case 1:} $x \ge \sqrt{\mu}.$
The function
$$\mathcal{H}_\nu(x)=(x^2 - \mu)^{1/4}  J_{\nu}(x) ,$$
as easy to check,
 satisfies the differential equation
 $$\mathcal{H}''_\nu(x)-\frac{\mu}{x(x^2-\mu)} \, \mathcal{H}'_\nu(x)+ \frac{4(x^2-\mu)^3+(6x^2-\mu)\mu}{4x^2 (x^2-\mu)^2} \, \mathcal{H}_\nu (x)=0. $$

Consider Sonin's function
$$ S(x)=
{\mathcal{H}}_{\nu}^2(x)+\frac{4x^2 (x^2-\mu)^2 }{4(x^2-\mu)^3+(6x^2-\mu)\mu} \, {\mathcal{H}'}_{\nu}^2(x), $$
then ${\mathcal{H}}_{\nu}^2 (x)\le S(x)$ for $x > \sqrt{\mu} >0.$

One finds
$$ S'(x)=\frac{24\mu x^3 (x^2-\mu)(4x^2+\mu)}{\left(4(x^2-\mu)^3+(6x^2-\mu)\mu \right)^2} \, {\mathcal{H}'}_{\nu}^2(x) \ge 0,$$
hence
$$|\mathcal{H}_{\nu}^2(x)|<\sqrt{\lim_{x \rightarrow \infty} S(x)}\, .$$
Using
$$J'_\nu(x)=\frac{J_{\nu-1}(x)-J_{\nu+1}(x)}{2} $$
and the asymptotic formula
$$J_\nu(x) \sim \sqrt{\frac{2}{\pi x}} \, \cos{(x-\frac{(2\nu+1)\pi}{4})} \, , $$
after some calculations one finds
$$ |\mathcal{H}_\nu(x)|< \sqrt{\frac{2}{\pi}} \, .$$
Since $\mathcal{H}_\nu^2(x)=S(x)$ at all local maxima the constant $\sqrt{\frac{2}{\pi}} $ is sharp.

\noindent
{\it Case 2:} $0 <x<\sqrt{\mu} , \; \;  \frac{1}{2} \le \nu \le 3.$
By (\ref{watsineq})  and $\nu >\frac{1}{2}$ we have
$$ \mathcal{H}_\nu(x)=( \mu -x^2 )^{1/4} J_\nu (x) \le
 \frac{( \mu -x^2 )^{1/4} x^\nu}{2^ \nu \Gamma(\nu+1)}\, .$$
 The maximum of the last expression is attained for $x= \sqrt{\nu^2-\nu/2} \, ,$
yielding
$$\mathcal{H}_\nu(x) \le \frac{\nu^{\nu/2} (2\nu-1)^{\nu/2+1/4}}{2^{(3\nu+1)/2} \, \Gamma(\nu+1)} :=\frac{f(\nu)}{\Gamma(\nu+1)} \, .$$
For $\nu >\frac{1}{2}$ the function $\Gamma( \nu+1)$ is increasing and it is easy to check that $f(\nu)$ is also increasing.
Hence for $\frac{1}{2} \le a \le \nu \le b, \; \; a \ne b,$ we have the following estimate
$$\mathcal{H}_\nu(x) < \frac{f(b)}{\Gamma (a+1)} \, .$$

This yields
$$\mathcal{H}_\nu(x) < \frac{3^{5/4}}{4 \sqrt{2} \, \Gamma(\frac{3}{2})}  = \frac{3^{5/4}}{2 \sqrt{2 \pi} }< \sqrt{\frac{2}{\pi}} \,, \; \; \; \frac{1}{2} \le \nu \le 2;$$
$$\mathcal{H}_\nu(x) < \frac{1.52}{\Gamma(3) } < \sqrt{\frac{2}{\pi}} \,, \; \; \;
2 \le \nu \le 2.6;$$
$$\mathcal{H}_\nu(x) < \frac{2.72}{\Gamma(3.6) } < \sqrt{\frac{2}{\pi}} \,, \; \; \;
2.6 \le \nu \le 3.$$

\noindent
{\it Case 3:} $0 <x<\sqrt{\mu} \,, \; \; \nu \ge 3.$
Inequality (\ref{ozmonj}) yields
$$ \mathcal{H}_\nu(x) <  \frac{2^{1/3} x^\nu (\mu-x^2)^{1/4}}{3^{2/3} \Gamma (\frac{2}{3}) \nu^{\nu+1/3}} \,
\exp \left(\frac{\nu^2-x^2}{2\nu+1} \right).$$
Let $r=(\nu/2)^{1/3},$
 by Lemma \ref{leftmax} we can set $$x= \nu \sqrt{1-(2\nu)^{-2/3} z} = \frac{r^2}{2} \,
 \sqrt{r^2-z}
 \,, \; \; \;  r^{-4} <z<1.$$
This gives
$$\mathcal{H}_\nu(x)<A \left( 1-\frac{z}{r^2} \right)^{r^3/4} \, (z \,r^4 -1)^{1/4} r^{-1}\,
\exp \left( \frac{z \, r^4}{4r^3+4}\right):=A f(z),$$
where $A= \frac{2^{1/6}}{3^{2/3} \Gamma (\frac{2}{3})} \, .$
We find
$$\frac{f'(z)}{f(z)}= \frac{(1+r^3+r^6-2z \,r^4-z^2 r^5 )r^3}{4(1+r^2)(r^2-z)(z\, r^4-1)},
$$
where
$$ 1+r^3+r^6-2z \,r^4-z^2 r^5 >(r^2-1)(1+r)(r^3-2r^2+r-1) >0,$$
for $\nu \ge 19/7.$  Hence $f(x)$ is increasing and
$$\mathcal{H}_\nu(x) <A f(1)= A \left( 1-\frac{1}{r^2} \right)^{r^3/4} \, (r^4 -1)^{1/4} r^{-1}\,
\exp \left( \frac{ r^4}{4r^3+4}\right) <$$
$$A e^{-r/4} (1-r^{-4})^{1/4} \exp \left( \frac{ r^4}{4r^3+4}\right) <A <\sqrt{\frac{2}{\pi}} \,.$$

This completes the proof.
\end{proof}

A similar but more complicated result can be given for $J'_\nu(x).$
\begin{theorem}
\label{thbesproizvod}
Let $\nu \ge \frac{1}{2}$ and $x \ge \nu+\frac{\sqrt{7}-1}{2^{2/3}} \, \nu^{1/3},$ then
\begin{equation}
\label{eqthbesproizvod}
\frac{x\left(4(x^2-\nu^2)^3-3x^4-10x^2 \nu^2+\nu^4 \right)^{1/4}}{x^2-\nu^2} \,| J'_\nu (x)|<
\frac{2}{\sqrt{\pi}} \, .
\end{equation}
\end{theorem}
\begin{proof}
Let
$$z(x)=\frac{x \psi^{1/4}(x)}{x^2-\nu^2} \, J'_\nu (x) ,$$
where
$$\psi(x)=4(x^2-\nu^2)^3-3x^4-10x^2 \nu^2+\nu^4.$$
First notice that $\psi(x) >0$ for $\nu \ge \frac{1}{2}$ and $x \ge \nu+\frac{\sqrt{7}-1}{2^{2/3}} \, \nu^{1/3}.$
Indeed, the substitutions
\begin{equation}
\label{subst7}
 x= y+ \nu+\frac{\sqrt{7}-1}{2^{2/3} }\,\nu^{1/3} , \; \; \;
\nu=(2^{-1/3}+n)^3,
\end{equation}
transforms $\psi$ into a polynomial in $n$ and $y$ with nonegative coefficients.
Consider Sonyn's function
$$S(x)=z^2(x)+\frac{x^2 (x^2-\nu^2)^2\psi^2(x)}{Q(x)}\,z'^2 (x), $$
and its derivative
$$S'(x)=\frac{6x^3(x^2-\nu^2)^4 \psi(x)P(x)}{Q^2(x)} \, z'^2(x), $$
where
$$Q(x)=16(x^2-\nu^2)^9-4(x^2-\nu^2)^6(2\nu^4-24 \nu^2 x^2-9x^4)+$$
$$(x^2-\nu^2)^3
(\nu^8+474\nu^4 x^4+440 \nu^2 x^6+45x^8)+ $$
$$
3x^2(2\nu^{10}+29\nu^8 x^2-112 \nu^6 x^4-66 \nu^4 x^6-42 \nu^2 x^8-3 x^{10}),
$$
and
$$P(x)=$$
$$16\nu^2 (x^2-\nu^2)^6 (4x^2+\nu^2)+8(x^2-\nu^2)^3 (5\nu^8+119\nu^6 x^2+313 \nu^4 x^4+145 \nu^2 x^6+6 \nu^8)+ $$
$$17\nu^{12}+618\nu^{10}x^2+695 \nu^8 x^4+492 \nu^6 x^6+2831\nu^4 x^8-54\nu^2 x^{10}+9 x^{12}- $$
$$2\nu^2(\nu^8+51 \nu^6 x^2+93 \nu^4 x^4+177 \nu^2 x^6-18x^8) .$$
Applying (\ref{subst7}) to $P$ and $Q$ we obtain polynomials with nonnegative coefficients.
Hence, $z^2(x) \le S(x)$ for $\nu \ge \frac{1}{2}$ and $x \ge \nu+\frac{\sqrt{7}-1}{2^{2/3}} \, \nu^{1/3}.$
Finally, using the asymptotics
$$J_\nu (x) \sim \sqrt{\frac{2}{\pi x} } \, \cos (x-\omega_\nu), \; \; \;
J'_\nu (x) \sim - \sqrt{\frac{2}{\pi x} } \, \sin (x-\omega_\nu),$$
we obtain
$$z^2(x) < \lim_{x \rightarrow \infty} S(x)= \frac{4}{\pi} \, ,$$
and the result follows.
\end{proof}

We'll need one more statement of this type for the Airy function
$$Ai(-x)=\frac{\sqrt{x}}{3} \, \,\left( J_{-1/3}(\zeta ) +
 J_{1/3}(\zeta ) \right) ,$$
  where $\zeta= \frac{2x^{3/2}}{3} \, .$

\begin{lemma}
\label{ozairy}
\begin{equation}
\label{eqozairy}
 (x+c )^{1/4} Ai(-x)<\frac{9}{14} \,, \; \; \; x \ge -c,
\end{equation}
where $c=15^{1/3} \cdot 2^{-4/3}.$
Moreover , the values of all local maxima of the function $(x+c )^{1/4} Ai(-x)$
are restricted to the interval $(\frac{1}{\sqrt{\pi}}\,,\frac{9}{14}).$
\end{lemma}
\begin{proof}
Consider the function
$$f(x)= (x+c)^{1/4} Ai(-x),\;\; x >-c,$$
which satisfies the following differential equation:
$$f''(x)-\frac{1}{2(x+c)}\, f'(x)+\left(x+\frac{5}{16(c+x)^2} \right)f(x)=0 .$$
The corresponding Sonin's function is
$$S(x)=f^2+\frac{f'^2}{x+\frac{5}{16(c+x)^2}} \,$$
and
$$S'(x)=- \frac{256 \, c \, x ( x+c)(x+2c)}{\left(16x(x+c)^2 +5\right)^2} \,.$$
Hence, $x=0$ is the only maximum of $S(x)$ for $x > -c$ and
$$ f(x) \le f(\xi)< \frac{9}{14} \, ,$$
where $\xi=1.12879717...,$ corresponds to the first maximum of $f(x).$
Using the asymptotic $M(-x) \sim \pi^{-1/2} x^{-1/4},$ where $M(x)$ is defined by
$Ai (x)= M(x) \sin \phi(x),$ one finds
$$\lim_{x->\infty} f(x)= \pi^{-1/2},$$ and the result follows.
\end{proof}

In the transition region the following inequality, which may be of independent interest, will be useful:
\begin{theorem}
\label{lemwron}
 For $x_1,x_2 \ge 0$ and $0 \le \nu \le \frac{1}{2},$
 \begin{equation}
\label{kernal}
\sqrt{x_1 x_2} \, \left| J_{-\nu}(x_1)J_{\nu}(x_2)-J_{-\nu}(x_2)J_{\nu}(x_1)\right| \le \frac{2}{\pi} \, \sin \pi \nu \, .
\end{equation}
\end{theorem}
\begin{proof}
The function
$$F=F(x_1,x_2)=\sqrt{x_1 x_2} \, \left( J_{-\nu}(x_1)J_{\nu}(x_2)-J_{-\nu}(x_2)J_{\nu}(x_1)\right)$$
satisfies the Bessel differential equations
\begin{equation}
\label{beseqs}
\frac{\partial^2} {\partial x_i^2} F+b(x_i)F=0, \; \; i=1,2;
\end{equation}
where
$$b(x)= 1+\frac{\frac{1}{4}-\nu^2}{x^2} >0 .$$
For $x >0$ we consider the following majorant of $F$ given by Sonin's function
$$F^2(x_1,x_2) \le S(x_1,x_2)=F^2 +\frac{\left(\frac{\partial} {\partial x_2} F \right)^2 }{b(x_2)}\, .$$
By (\ref{beseqs}) we have
$$\frac{\partial} {\partial x_2} S = - \frac{\frac{\partial } {\partial x_2}\, b(x_2)}{b^2(x_2)} \left(\frac{\partial} {\partial x_2} F \right)^2.$$
Since
$$\frac{\partial }{\partial x_2}\, b(x_2)= \frac{4\nu^2-1}{2x_2^3} <0$$
for $\nu < \frac{1}{2}\, ,$
the Sonin's function is increases in  $x_2.$
Using the asymtotics
$$ J_\nu (x)= \sqrt{\frac{2}{\pi x}}\, \cos (x-\frac{(2 \nu+1)\pi}{4} \, )+o(x^{-1/2}),$$
$$ J'_\nu (x)= -\sqrt{\frac{2}{\pi x}} \, \sin (x-\frac{(2 \nu+1)\pi}{4} \, )+o(x^{-1/2}),$$
and inequalities
(\ref{besszegj}),(\ref{besszegy}),
on taking the limit we obtain
$$F^2(x_1,x_2) \le S (x_1,x_2) \le \lim_{ x_2 \rightarrow \infty}  S (x_1,x_2)=$$
$$
\frac{2x_1}{\pi}\left(J_\nu^2(x_1)-2\cos \pi \nu J_\nu(x_1)J_{-\nu}(x_1)+ J_{-\nu}^2(x_1)\right)=$$
$$ \frac{2x_1 \sin^2 \pi \nu}{\pi}\, \left(J_\nu^2(x_1)+Y_\nu^2(x_1) \right) \le
\frac{4 \sin^2 \pi \nu}{\pi^2}\,.$$
This completes the proof.
\end{proof}

To get better numerical constants the following inequalities will be useful.
\begin{lemma}
\label{integral}
For $x \ge 0,$
\begin{equation}
\label{integralsq}
 \int_0^\infty \frac{\sin^2 t}{(t+x)^2}\, dt < \frac{1}{2 x} \, ,
\end{equation}
\begin{equation}
\label{integralmod}
 \int_0^\infty \frac{|\sin t|}{(t+x)^2}\, dt < \frac{2}{\pi x} \, .
\end{equation}
\end{lemma}
\begin{proof}
We have
$$ \int_0^\infty \frac{\sin^2 t}{(t+x)^2}\, dt =\frac{1}{2x}-\frac{1}{2}\int_0^\infty \frac{\cos 2 t}{(t+x)^2}\, dt <\frac{1}{2x}$$
 because
 $$\int_0^\infty \frac{\cos 2 t}{(t+x)^2}\, dt =
 \sum_{k=0}^\infty \int_{\frac{\pi k}{2}}^{\frac{\pi(k+1)}{2}} \frac{\cos 2 t}{(t+x)^2}\, dt >0 $$
  is an alternating sum with decreasing terms. This proves (\ref{integralsq}).

Now we prove (\ref{integralmod}). Let
$$f(x)= x \int_0^\infty \frac{|\sin t|}{(t+x)^2}\, dt .$$
We have
$$f(x) <  \int_0^1 \frac{x \, t}{(t+x)^2}\, dt +  \int_1^\infty \frac{x \, dt}{(t+x)^2} \, dt = x \ln \frac{x+1}{x} < \frac{2}{\pi} \, , \; \; \;  x \in [0, \frac{3}{4} \,] .$$
Next, we show that
$$\lim_{x \rightarrow \infty} f(x)= \frac{2}{\pi} \, .$$
Using
$|\sin ( x+\pi k)|=\sin x,\; \; 0 \le x \le \frac{\pi}{2} \, ,$ we obtain
$$ f(x)=x \sum_{k=0}^\infty \int_{\pi k}^{\pi(k+1)} \frac{|\sin t|}{(t+x)^2} \, dt =
x \sum_{k=0}^\infty \int_0^\pi \frac{\sin t}{(t+x+\pi k)^2} \, dt=$$
$$x \int_0^\pi \sin t \sum_{k=0}^\infty \frac{1}{(t+x+\pi k)^2} \, dt=
\frac{x}{\pi^2} \int_0^\pi \psi'(\frac{t+x}{\pi})\, \sin t \, dt,$$
where $\psi(x)=\Gamma'(x)/\Gamma(x)$ is the digamma function.
Since $\psi'(x)=x^{-1}+O(x^{-2})$
for $x \rightarrow \infty$ this yields
$$f(x)= \frac{x}{\pi} \int_0^\pi \frac{\sin t}{t+x} \, dt +O(x^{-1})=\frac{x}{\pi} \int_0^\pi \frac{\sin t}{x} \, dt +O(x^{-1})= \frac{2}{\pi} +O(x^{-1}).$$
Now it is enough to show that $f(x)$ is increasing for $x >\frac{3}{4} \,.$
Using the inequalities \cite{ronning},
$$\frac{1}{1+x}+ \frac{1}{x^2} < \psi'(x) < \frac{1}{1+x}+ \frac{1}{x^2}+\frac{1}{(x+1)^2} \, ,$$
we have
$$\pi^2 f'(x)= \frac{d}{d x} \, \left( x \int_0^\pi \psi'(\frac{t+x}{\pi})\, \sin t \, dt,\right)=$$
$$\int_0^\pi \psi'(\frac{t+x}{\pi})\, \sin t \, dt +\frac{x}{\pi}\int_0^\pi \psi''(\frac{t+x}{\pi})\, \sin t \, dt = $$
$$\int_0^\pi \psi'(\frac{t+x}{\pi})\, \sin t \, dt -\frac{x}{\pi}\int_0^\pi \psi'(\frac{t+x}{\pi})\, \cos t \, dt  >$$
$$\int_0^\pi \left(\frac{\pi}{t+x+\pi}+\frac{\pi^2}{(t+x)^2} \right) \sin t \, dt \; -$$
$$
\int_0^{\pi/2} \left(\frac{x}{t+x+\pi}+\frac{\pi x}{(t+x)^2}+\frac{\pi x}{(t+x+\pi)^2} \right) \cos t \, dt \; -$$
$$\int_{\pi/2}^\pi \left(\frac{x}{t+x+\pi}+\frac{\pi x}{(t+x)^2} \right) \cos t \, dt >$$
$$\sum_{j=0}^3 \left(\frac{\pi}{x+\frac{(j+5) \pi}{4}}+\frac{\pi^2}{\left( x+\frac{(j+1) \pi}{4} \, \right)^2} \right)\int_{j\pi/4}^{(j+1)\pi/4}  \sin t \, dt \; -$$
$$
x \left(\frac{1}{x+\pi}+\frac{\pi}{x^2}+\frac{\pi}{(x+\pi)^2} \right) \int_0^{\pi/2} \cos t \, dt \; - x \left(\frac{1}{x+2\pi}+\frac{\pi}{(x+\pi)^2} \right)\int_{\pi/2}^\pi \cos t \, dt \; = $$
$$\frac{ P(x)}{2 Q(x)},$$
where
$$\pi^{-11}P(x)=$$
$$114688y^{11}+4096 (259-2 \sqrt{2} \,) y^{10} +1024 (4216-77 \sqrt{2}\,)y^9+
1536(6641-244 \sqrt{2}\,)y^8+$$
$$64(242859-16954 \sqrt{2}\,)y^7+400(40011-5002 \sqrt{2}\,)y^6+
4(2820158-591757 \sqrt{2}\,)y^5+$$
$$4(1332737-443447 \sqrt{2}\,)y^4 +(1564982-808187 \sqrt{2}\,)y^3+
5(44796-40571 \sqrt{2}\,)y^2-$$
$$3(1140+7109 \sqrt{2}\,)y-3780,$$
and
$$Q(x)=y(y+1)^2 (2y+1)^2 (4y+1)^2 (4y+3)^2(y+2)(2y+3)(4y+5)(4y+7),\; \; y=x/\pi .$$
It is easy to check that the polynomial $P(x+\frac{2}{3} \,)$ has only positive coefficients.
Hence $f'(x)>0$ for $x >2/3,$ and $f(x) < \lim_{x \rightarrow \infty}f(x)=\frac{2}{\pi} \, .$
This competes the proof.
\end{proof}
\section{Oscillatory Region}
\label{secosc}
Having at hand an upper bound on $|J_\nu(x)|$ one can estimate the difference
\begin{equation}
\label{defr}
r(x)= \sqrt{\frac{\pi x}{ 2}} \, J_\nu (x)-\cos (x-\omega_\nu),
\end{equation}
in a rather elementary way.
Notice that $r(x)$
satisfies the following differential equation
$$r''+r=\sqrt{\frac{\pi }{ 2 x^3}} \, \left(\nu^2- \frac{1}{4} \right) J_\nu (x),$$
with the general solution of the form
\begin{equation}
\label{fomular}
r(x)=c_1 \cos x+c_2 \sin x + \sqrt{\frac{\pi }{ 2 }} \, \left( \frac{1}{4}-\nu^2 \right)\int_{x}^\infty  \frac{\sin(t-x)}{t^{3/2}} \,J_\nu (t) dt.
\end{equation}
Now one has only to estimate the integral and to notice that as far as it is $o(1),$
we have $c_1=c_2=0$ by an obvious limiting argument.

In \cite{olenko} Olenko proved the inequalities which for $\nu >0$ can be written as
 $$ c_1 \nu^{7/6} \le \sup_{x \ge 0} x^{3/2}\left| J_\nu(x)-\sqrt{\frac{2}{\pi x}}  \, \cos (x-\omega_\nu )\right| \le c_2 \nu^{13/6},$$
 with explicit constants $c_1,c_2,$
 and raised the question what is the best possible exponent $\alpha$ of $\nu$ in these inequalities.
 It turns out that that the answer $ \alpha=2$ can be readily extracted from (\ref{fomular}), since
starting with a reasonably sharp approximation to the Bessel function one can iterate
it getting more and more accurate yet more complicate approximations.

\begin{theorem}
\label{difoz12}
\begin{equation}
\label{eqdifoz12}
J_\nu(x) = \sqrt{\frac{2}{\pi x}} \, \cos(x- \omega_\nu) +  \theta c \mu x^{-3/2},
\end{equation}
where
$$
c=\left\{
\begin{array}{ccc}
\left(\frac{2}{\pi }\right)^{3/2} \,  ,& x \ge 0, & |\nu| \le \frac{1}{2},\\
& & \\
 \frac{\sqrt{2}}{2} \, , & x \ge \sqrt{\mu}, & \nu > \frac{1}{2} \, ,\\
&&\\
 \frac{5 }{4} \, , & 0 <x < \sqrt{\mu}, & \nu > \frac{1}{2} \, .
\end{array}
\right.
$$
Moreover, up to the numerical factor $c$ the error term in (\ref{eqdifoz12}) is sharp. In particular, $c$ cannot be taken less than  $1/\sqrt{2\pi} \,.$
\end{theorem}
\begin{proof}
To estimate the integral in (\ref{fomular}) for $|\nu| \le \frac{1}{2}$ we apply (\ref{integralmod}) yielding
$$\mathcal{I}_\nu (x):= \sqrt{\frac{\pi}{2}} \,\left| \int_{x}^\infty  \frac{\sin(t-x)}{t^{3/2}} \,J_\nu (t) dt \right| \le
\int_{x}^\infty  \frac{|\sin(t-x)|}{t^{2}} \; dt   < \frac{2}{\pi x}.$$
Thus
$$r(x)=c_1 \cos x+c_2 \sin x + \theta \, \frac{2\mu}{\pi x} \, .$$
For $\nu > \frac{1}{2} $ and $x > \sqrt{\mu} \,$ we use (\ref{eqbess}) and the inequality $ \arcsin x \le \frac{\pi x}{2} \, .$ This gives
$$\mathcal{I}_\nu (x) \le
\int_{x}^\infty  \frac{|\sin(t-x)|}{t^{3/2}(t^2-\mu)^{1/4}} \, dt \le\sqrt{\int_x^\infty  \frac{\sin^2 (t-x) \, dt}{t^2} \, \cdot  \int_x^\infty  \frac{dt}{t \sqrt{t^2-\mu}}} $$
$$=
\sqrt{\frac{\arcsin \frac{\sqrt{\mu}}{x}}{2x \sqrt{\mu}}} \le \frac{\sqrt{\pi}}{2 x} \, .$$
Hence in this case
$$r(x)=c_1 \cos x+c_2 \sin x + \theta \,  \frac{\sqrt{\pi} \, \mu}{2 x} \, .$$

Similarly, for $\nu > \frac{1}{2} $ and $0 < x \le \sqrt{\mu} \,,$
$$\mathcal{I}_\nu (x) \le  \int_{\sqrt{\mu}}^\infty  \frac{|\sin(t-x)|}{t^{3/2}(t^2-\mu)^{1/4}} \, dt+
\int_{x}^{\sqrt{\mu}}  \frac{|\sin(t-x)|}{t^{3/2}(\mu-t^2)^{1/4}} \, dt \le
$$
$$
\frac{1}{2} \, \sqrt{\frac{\pi}{\mu}} + \int_{x}^{\sqrt{\mu}}  \frac{|\sin(t-x)|}{t^{3/2}(\mu-t^2)^{1/4}} \, dt \le \frac{1}{2} \, \sqrt{\frac{\pi}{\mu}} +\sqrt{\int_x^{\sqrt{\mu}}  \frac{ dt}{t^2} \, \cdot  \int_x^{\sqrt{\mu}}  \frac{dt}{t \sqrt{\mu-t^2}}}=
$$
$$
\frac{1}{2} \, \sqrt{\frac{\pi}{\mu}} + \sqrt{\frac{\sqrt{\mu}-x}{\mu x} \,
\ln \frac{\sqrt{\mu}+\sqrt{\mu-x^2}}{x}} < \frac{\sqrt{\pi}}{2x}+
\frac{\sqrt{x(\sqrt{\mu}-x)\ln \frac{2\sqrt{\mu}}{x}}}{x \sqrt{\mu}} \, .
$$
An elementary investigation shows that the maximum of the function
$$x(\sqrt{\mu}-x)\ln \frac{2\sqrt{\mu}}{x} $$
is attained for $x=d \sqrt{\mu} \,, \; \; d =0.314711...,$ and does not exceed $2 \mu/5.$
Hence,
$$\mathcal{I}_\nu (x) < \frac{\sqrt{\pi}+2 \sqrt{\frac{2}{5}}}{2x} < \frac{6 \sqrt{\pi}}{7x} \, ,$$
and
$$r(x)= c_1 \cos x+c_2 \sin x + \theta \,  \frac{6 \sqrt{\pi} \, \mu}{7 x} \, .$$
Finally, since for $x \rightarrow \infty ,$
$$J_\nu (x)= \sqrt{\frac{2}{\pi x}} \, \cos (x-\omega_\nu ) +O(x^{-3/2} ),$$
that is
$\lim_{x \rightarrow \infty} r(x)=0,$ we conclude that $c_1=c_2=0$ and (\ref{eqdifoz12}) follows.

Let us show now that up to the numerical factor $c$ the error term in (\ref{eqdifoz12}) is sharp.

By (\ref{fomular}) and  (\ref{eqdifoz12}) we have
$$\mathcal{R}_\nu(x)=x^{3/2} \left| J_\nu(x)-\sqrt{\frac{2}{\pi x}}  \, \cos (x-\omega_\nu )\right|=\mu x \left| \int_x^\infty \frac{\sin(t-x)}{t^{3/2}} \, J_\nu(t) dt \right|=$$
$$
\sqrt{\frac{2}{\pi}}\,\mu x \left| \int_x^\infty \frac{\sin(t-x)\cos(t-\omega_\nu)}{t^2} \, dt \right|+ \theta c \mu x \left| \int_x^\infty \frac{\sin(t-x)}{t^3} \, dt \right|:=
 \sqrt{\frac{2}{\pi}}\,\mu x | I_1 |+  I_2.
$$
Here
$$ I_2 = \theta c \mu x  \int_x^\infty \frac{dt}{t^3}=\frac{\theta c \mu}{2x} \, .  $$
To bound $I_1$ we introduce two auxiliary functions $f$ and $g$ (see \cite[Ch. 6]{olv}), defined by
$$Si(x)= \frac{\pi}{2} -f(x)\cos x-g(x) \sin x,$$
$$Ci(x)= f(x)\sin x-g(x) \cos x,$$
with the asymptotics
$$f(x) =\frac{1}{x}+O(x^{-3}), \; \; \; g(x) =\frac{1}{x^2}+O(x^{-4}), \; \; x \rightarrow \infty.$$
Calculations yield
$$I_1= \left. Si(2z) \sin(x+\omega_\nu)+Ci(2z) \cos(x+\omega_\nu ) -\frac{\sin t \cos(z-\omega_\nu )}{z} \right|_{t=0}^\infty =$$
$$
\left. f(2z) \sin(2t+x-\omega_\nu)- g(2z) \cos(2t+x-\omega_\nu)-
\frac{\sin t \cos(z-\omega_\nu)}{z}\right|_{t=0}^\infty =
$$
$$
\frac{\sin(\omega_\nu-x)}{2x}+O(x^{-2}), \; \; \;z=x+t.
$$
Hence
$$\mathcal{R}_\nu (x)=\frac{\mu}{\sqrt{2\pi}} \,| \sin(\omega_\nu-x)| +O(x^{-1}),$$
and the result follows.
\end{proof}

Applying (\ref{eqdifoz12}) to the Airy function yields the following approximation:
\begin{corollary}
\label{corairy}
For $x >0,$
\begin{equation}
\label{ozairy1}
Ai(-x) = \frac{\cos( \zeta-\frac{\pi}{4} \,)}{\sqrt{\pi} \, x^{1/4}}+\theta \, \frac{5}{6 \sqrt{3} \, \pi^{3/2} x^{7/4}} \,;
\end{equation}
\end{corollary}

\section{Transition Region}
\label{sectran}
Our estimates in the transition region are based on the following simple observation.
\begin{lemma}
\label{lemmain2}
Let for a real constant $q,$
$$y''(x)+q^2 x y(x)=u(x). $$
Then for $x \ge 0,$
\begin{equation}
\label{eqtrans}
y(x)=\sqrt{x} \left(c_1  J_{-1/3}(\frac{2 q x^{3/2}}{3} )+c_2  J_{1/3}(\frac{2 q x^{3/2}}{3} ) \right)+
\end{equation}
$$\theta \, q^{-1}  x^{-1/4} \int_0^x |u(t)|t^{-1/4}  dt ,$$
provided the integral exists.
\end{lemma}
\begin{proof}
Let $y_1$ and $y_2$ be two linearly independent solutions of the homogeneous equation
$$y''(x)+q^2 x y(x)=0,$$
and let
$$U(x,t)=\frac{y_1(t)y_2(x)-y_1(x)y_2(t)}{y_1(t)y'_2 (t)-y'_1(t) y_2(t)} \, .$$
Then
$$y(x)=c_1 y_1(x)+c_2 y_2(x)+\int^x U(x,t)u(t) dt, $$
and choosing
$$y_1(x)= \sqrt{x}\, J_{-1/3}(\frac{2 q x^{3/2}}{3} ), \; \; \;
y_2(x)= \sqrt{x}\, J_{1/3}(\frac{2 q x^{3/2}}{3}),$$
we find
$$y_1(t)y'_2 (t)-y'_1(t) y_2(t) =\frac{3 \sqrt{3}}{2 \pi} .$$
Applying (\ref{kernal}) we obtain
$$U(x,t)= 2 \pi \sqrt{\frac{t x}{27}} \left(
J_{-1/3}(\frac{2 q t^{3/2}}{3} ) J_{1/3}(\frac{2 q x^{3/2}}{3} )-J_{-1/3}(\frac{2 q x^{3/2}}{3} )J_{1/3}(\frac{2 q t^{3/2}}{3} )\right)\le$$

$$  q^{-1} t^{-1/4} x^{-1/4}. $$
Hence
$$\left|\int_0^x U(x,t)u(t)dt \right| \le q^{-1}x^{-1/4} \int_0^x |u(t)| t^{-1/4}  dt ,$$
and the result follows.
\end{proof}

To find the constant of integration we use the value of $J_\nu(x)$ at two points:
$x=\nu$ and at the first zero $x=j_{\nu 1}.$
The following bounds are known \cite{quwong} (see \cite{finch} for a review of recent results),
\begin{equation}
\label{firstzerbes}
 j_{\nu s} = \nu+ 2^{-1/3} a_s \nu^{1/3}+\theta^2 \, \frac{3 \cdot 2^{-2/3} a_s^2}{10} \,\nu^{-1/3} , \; \; \; \nu >0,
\end{equation}
where $a_s$ is $s^{th}$ positive zero of
the Airy function
$Ai(-x).$

In the sequel we set
$\gamma= 2^{-1/3} a_1 =1.855757...$
  Thus, we can write
\begin{equation}
\label{pribfirstzerbes}
j_{\nu 1}= \nu+\gamma \nu^{1/3}+\theta^2 \, \frac{3 \gamma^2}{10} \,\nu^{-1/3}.
\end{equation}

\begin{lemma}
\label{tailor}
\begin{equation}
\label{znachroot}
  J_\nu (\nu+\gamma \nu^{1/3}) < \frac{7}{6 \nu} \,, \; \; \nu \ge \frac{1}{2} \, .
\end{equation}
\end{lemma}
\begin{proof}
Let $j=j_{\nu 1},$
 by (\ref{pribfirstzerbes}) we have
$$ J_\nu (\nu+\gamma \nu^{1/3})=- \theta^2 \, \frac{3 \gamma^2}{10} \,\nu^{-1/3} J'_\nu (\nu+\gamma \nu^{1/3}+\theta_1^2 \, \frac{3 \gamma^2}{10} \,\nu^{-1/3}).$$
Setting
$$ \gamma \nu^{1/3}+\theta_1^2 \, \frac{3 \gamma^2}{10} \,\nu^{-1/3}:=\delta \nu^{1/3}, \;\;\;
\nu=(\delta/\epsilon)^{3/2},$$
by (\ref{eqthbesproizvod}) and $\gamma > \frac{\sqrt{7}-1}{2^{2/3}}$ we obtain
$$|J'_\nu (\nu+\delta \nu^{1/3}|< \frac{2 \epsilon}{\sqrt{\pi}} \,\phi(\epsilon, \delta),$$
where
$$\phi(\epsilon, \delta)= \frac{2+\epsilon}{1+ \epsilon}  \, \left(
4(8\delta^3-3)+16(3 \delta^3-2) \epsilon+4(6 \delta^3-7) \epsilon^2+4(\delta^3-3) \epsilon^3-3 \epsilon^4 \right)^{-1/4}.$$
Notice that $\phi(\epsilon, \delta) <\phi(\epsilon, \gamma)$
since
$\delta > \gamma$
and
$$\frac{\partial}{\partial \delta}\, \phi(\epsilon, \delta)= - \frac{3 \delta^2 (1+\epsilon)^4}{2+\epsilon} \,\phi^5(\epsilon, \delta) <0.$$
Moreover,
$$ \epsilon = \delta \nu^{-2/3}= \gamma \nu^{-2/3}+\frac{3 \gamma^2}{10}\, \nu^{-4/3}<6,$$
whereas
$$\frac{\partial}{\partial \epsilon}\, \phi(\epsilon, \gamma)<0, \;\;\; 0 \le \epsilon \le 7. $$
Thus, $\phi(\epsilon, \delta) <\phi(0, \gamma),$ and we get
$$|J'_\nu (\nu+\delta \nu^{1/3}|< \frac{2 \epsilon}{\sqrt{\pi}} \,\phi(0, \gamma)=\frac{2^{3/2} \gamma}{\sqrt{\pi}\,
(8\gamma-3)^{1/4}} \, \nu^{-2/3},$$
and the result follows.
\end{proof}

Applying Lemma \ref{lemmain2} to $J_\nu$ we obtain
\begin{theorem}
\label{thbessel}
Let $\nu \ge \frac{1}{2} \, ,$ and
let $x=\nu+\nu^{1/3} z, $ then
\begin{equation}
\label{eqbes2}
J_\nu(x)= \frac{2^{1/3}}{ \sqrt{\nu^{2/3}+z} }\, Ai(-2^{1/3} z) +\theta \frac{ 23  \max\{1,z^{9/4}\} }{2 \nu^{2/3} \sqrt{\nu^{2/3}+z}} \,.
\end{equation}
\end{theorem}
\begin{proof}
Consider the function
$$y(z)= \sqrt{\nu+\nu^{1/3} z} \, J_\nu (\nu+\nu^{1/3} z) ,$$
which, as easy to check, satisfies the differential equation
$$y''(z)+2z y(z)= \frac{8z^3+12 \nu^{2/3} z^2-1}{4(\nu^{2/3}+z)^2} \, y(z). $$
Lemmas \ref{lemmain2} and \ref{bessel} yield
\begin{equation}
\label{asympty}
y(z)= \sqrt{z} \left(c_1 J_{-1/3} ( \sqrt{2} \,\zeta)+
c_2 J_{1/3} ( \sqrt{2} \,\zeta)\right) +\theta R(z),
\end{equation}
where $\zeta=\frac{2z^{3/2}}{3}\, ,$ and
$$z^{1/4} R(z)= \frac{1}{ 4\sqrt{2}  } \int_0^z \left| \frac{8t^3+12 \nu^{2/3} t^2-1}{(\nu^{2/3}+t)^2 t^{1/4}}
 \, y(t)\right| dt \le $$
$$\frac{1}{4\sqrt{2} } \int_0^z  \frac{|8t^3+12 \nu^{2/3} t^2-1|}{(\nu^{2/3}+t)^{3/2} t^{1/4}} \cdot (x^2-\mu)^{-1/4} \sqrt{\frac{2 x}{\pi}} \, dt = $$
$$\frac{\nu^{1/6}}{2\sqrt{2 \pi} } \int_0^z
\frac{|8t^3+12 \nu^{2/3} t^2-1|}{(\nu^{2/3}+t)^{3/2} (4 \nu^{2/3} t^2+8 \nu^{4/3} t+1)^{1/4}t^{1/4}} \, dt <$$
$$\frac{1}{\sqrt{\pi}  } \, \max \left\{\int_0^z
\frac{(3\nu^{2/3}+2t)t^{3/2} dt}{(\nu^{2/3}+t)^{3/2}(2 \nu^{2/3}+t)^{1/4}} \,,
 \int_0^z
\frac{dt }{4 \sqrt{t} \,(\nu^{2/3}+t)^{3/2}(2 \nu^{2/3}+t)^{1/4}} \right\} <$$
$$
\frac{1}{\sqrt{\pi}  } \, \max \left\{\int_0^z
\frac{(3\nu^{2/3}+2t)t^{3/2} dt}{(\nu^{2/3}+t)^{3/2}(2 \nu^{2/3}+t)^{1/4}} \,,
 \int_0^z
\frac{dt }{4 \cdot 2^{1/4} \nu^{7/6} \sqrt{t} } \right\}.
$$
For $t >0$ the function
$$\frac{3\nu^{2/3}+2t }{(\nu^{2/3}+t)^{3/2}(2 \nu^{2/3}+t)^{1/4}} $$
is decreasing in $t$ and therefore is less than $3 \cdot 2^{-1/4} \nu^{-1/2},$ its value at $t=0$. Thus, we obtain
\begin{equation}
\label{esimaterr}
 R(z) <
  \frac{1}{\sqrt{\pi} \, z^{1/4}} \max \left\{ \int_0^z
 \frac{3t^{3/2}}{2^{1/4}  \sqrt{\nu}} \; dt , \frac{\sqrt{z}}{2^{5/4} \nu^{7/6}}\right\}
  =
  \end{equation}
  $$\frac{z^{3/4}}{2^{5/4} \sqrt{\pi \nu}} \, \max \left\{\frac{6}{5} \, z^{3/2} , \nu^{-2/3} \right\} .$$

It is left to find the constants $c_1,c_2.$
For $x=\nu,$ that is for $z=0,$ we have $R(z)=0,$ and comparing (\ref{asympty}) with (\ref{besnunu}) we get
$$y(0)=\lim_{z \rightarrow 0} \sqrt{z} \left(c_1 J_{-1/3} ( \sqrt{2} \,\zeta)+
c_2 J_{1/3} ( \sqrt{2} \,\zeta)\right)= c_1 \lim_{z \rightarrow 0}
 \sqrt{z} \cdot \frac{3^{1/3}}{2^{1/6} \Gamma(\frac{2}{3} \,) \sqrt{z}}=$$
 $$
 \frac{3^{1/3} c_1}{2^{1/6} \Gamma(\frac{2}{3} \,)} =  \sqrt{\nu} J_\nu(\nu)=
 \frac{2^{1/3} \sqrt{\nu}}{3^{2/3}\Gamma(\frac{2}{3} \,) (\nu+\theta^2 \alpha)^{1/3}} \, ,$$
 hence
 $$c_1= \frac{\sqrt{2\nu}}{3(\nu+\theta^2 \alpha)^{1/3}} =
 \frac{\sqrt{2} \, \nu^{1/6}}{3} \left(1-\theta^2
 \frac{ \alpha}{3 \nu} \right). $$

Thus, using this and setting $c_3=c_2-c_1,$ we can write
\begin{equation}
\label{firstc}
y(z)= 2^{1/3} \nu^{1/6}\, Ai(-2^{1/3} z) -
\end{equation}
$$ \frac{\theta_1^2}{25 \nu^{5/6}} \, Ai(-2^{1/3} z) +
c_3 \sqrt{z} \, J_{1/3} ( \sqrt{2} \,\zeta)+\theta_2 R(z) :=2^{1/3} \nu^{1/6}\, Ai(-2^{1/3} z)+\mathcal{R}.
$$
Now $Ai(-2^{1/3} \gamma)=0$  and by Lemma \ref{tailor} we have
$$\frac{7 \sqrt{\nu+\gamma \nu^{1/3}}}{6 \nu}  > y(\nu+\gamma \nu^{1/3})=
c_3 \sqrt{\gamma} \, J_{1/3} ( \frac{(2\gamma)^{3/2}}{3})+\theta_2 R(\gamma)>
\frac{7}{20} \, c_3 - \frac{7}{3 \sqrt{\nu }} \,,
$$
yielding
$$
c_3 < \frac{10}{3 \sqrt{\nu}} \, (2+\sqrt{1+\gamma \nu^{-2/3}}) < \frac{40}{3 \sqrt{\nu}} \,.
$$
On the other hand,
$$ 0 < y(\nu+\gamma \nu^{1/3})<c_3 \sqrt{\gamma} \, J_{1/3} ( \frac{(2\gamma)^{3/2}}{3})+
R(\gamma) ,$$
giving
 $c_3 > -7/\sqrt{\nu}  \, .$

Inequality (\ref{eqozairy} ) gives
$$|Ai(-2^{1/3} z)| < \frac{9}{7 \cdot 2^{7/12}(4z+30^{1/3})^{1/4}}\,,
$$
and using (\ref{watsineq}) and (\ref{besszegj}) we have
$$\sqrt{z} \, |J_{1/3} (\sqrt{2} \, \zeta) | <\min \{\frac{2^{1/6} z}{3^{1/3}
\Gamma(\frac{4}{3})} \,, \frac{\sqrt{3}}{2^{1/4} \sqrt{\pi} \, z^{1/4}}\} \, .$$
Combining these estimates with (\ref{esimaterr}) after some straightforward calculations one finds
$$|\mathcal{R}|< \frac{ 23  \max\{1,z^{9/4}\} }{2 \sqrt{\nu}} \, ,$$
and (\ref{eqbes2}) follows.
\end{proof}


\section{Sharper asymptotics}
\label{secanother}
The classical asymptotic given by (\ref{eqdifoz12}) does not makes much sense  for $x = O(\mu)$
when the main term and the error are of the same order. Here using formula (\ref{eqwkb}) we derive a different asymptotic expression with much smaller error term. It also leads to
very sharp approximation of the Airy function $Ai(-x).$

We'll need a few lemmas given in \cite{kras10}.
\begin{lemma}
\label{lemmain1} Let $f(x)$ satisfy the differential equation
$$f''(x)+b^2(x)f(x)=0,$$
where $b(x)>0$ and $b''(x)$  exists on an interval $\mathcal{I}.$\\
Let $g(x)=\sqrt{b(x)} \, f(x),$\\
then for $x \in I,$ provided the integral exists,
\begin{equation}
\label{eqozosc}
g(x)=c_1 \cos{\mathcal{B}(x)}+c_2 \sin{\mathcal{B}(x)} + \theta \int_a^x
 \left| \frac{3 b'^2(t)-2b(t)b''(t)}{4b^3(t)} \, g(x) \right| dt,
\end{equation}
where
$\mathcal{B}(x)=\int^x b(t)dt,$ $a \in I$ is arbitrary and $|\theta|\le 1.$
\end{lemma}
\begin{proof}
Observe that $g(x)$ satisfies the equation
\begin{equation}
\label{epsileq1}
g''-\frac{b'}{b} \, g' + g \left(b^2 -\epsilon \right)=0, \; \; \;\epsilon= \epsilon(x)=\frac{2b b''-3b'^2}{4b^2} \, .
\end{equation}
The solution of the corresponding homogeneous equation
$$g_0''-\frac{b'}{b} \, g_0' +b^2 g_0=0$$
is
$$g_0=c_1\sin{ \mathcal{B}(x)}+ c_2\cos{ \mathcal{B}(x)}.$$
Solving formally (\ref{epsileq1}) as a nonhomogeneous equation with the right hand side
$\epsilon(x)g(x)$ we get
$$ g(x)=g_0(x)+\int_a^x \frac{\sin{( \mathcal{B}(x)-\mathcal{B}(t))}}{b(t)} \, \epsilon (t) g(t)dt=$$
$$
g_0(x)+\theta \int_a^x \left|\frac{\sin{( \mathcal{B}(x)-\mathcal{B}(t))}}{b(t)} \, \epsilon (t) g(t)\right|dt =$$
$$g_0(x)+\theta \int_a^x
\left| \frac{3 b'^2(t)-2b(t)b''(t)}{4b^3(t)} \, g(x)\right| dt.
$$
\end{proof}

The normal form of differential equation (\ref{difeqbes}) is

$$f''+(1-\frac{\nu^2-\frac{1}{4}}{x^2}) f=0, \; \;
f=\sqrt{x} \, J_\nu (x).$$
Thus, for
$x >  \sqrt{\max\{0,\nu^2-\frac{1}{4}\, \}}$
we have
$$b(x)=\frac{\sqrt{x^2-\nu^2+\frac{1}{4}}}{x} \, ,$$
$$\mathcal{B}(x)=  \left\{
\begin{array}{cc}
\sqrt{x^2+\mu}+ \sqrt{\mu} \, \ln { \frac{x}{\sqrt{\mu} +\sqrt{\mu+x^2}} \,}  , & |x|\le \frac{1}{2} \, ,\\
&\\
\sqrt{x^2-\mu}+ \sqrt{\mu} \, \arcsin{ \frac{\sqrt{\mu}}{x}} \, , & x \ge \frac{1}{2}\, ,
\end{array}
\right.
$$
and
$$g(x)= (x^2-\nu^2+\frac{1}{4}\,)^{1/4} \, J_\nu (x). $$

\begin{theorem}
\label{thbesse0}
For $|\nu |\le \frac{1}{2}$ and $x >0,$
\begin{equation}
\label{eqbes0}
J_\nu(x)=
\sqrt{\frac{2}{\pi}} \, (x^2+\mu)^{-1/4} \cos{ \left(\mathcal{B}(x) - \omega_\nu \right) } + \theta \frac{\mu}{\sqrt{2 \pi x} \,(x^2+\mu)^{3/2} } \, .
\end{equation}
For $|\nu | > \frac{1}{2}$ and $x >\mu,$
\begin{equation}
\label{eqbes1}
J_\nu(x)=
\sqrt{\frac{2}{\pi}} \, (x^2-\mu)^{-1/4} \cos{\left(\mathcal{B}(x) - \omega_\nu \right)} + \theta  \frac{13 \mu}{12 \sqrt{2\pi} \, (x^2-\mu)^{7/4}} \, .
\end{equation}
\end{theorem}
\begin{proof}
  Since $|J_\nu(x)| \le \sqrt{\frac{2}{ \pi x}} \, $ for $|\nu| \le \frac{1}{2}\,,$ by (\ref{eqozosc}) we have
$$g(x)=g_0 + \theta \, \frac{\mu}{\sqrt{8 \pi}} \int_x^\infty \frac{6z^2+\mu}{z^{3/2} \, (z^2+\mu)^{9/4}} \, dz =g_0+ \theta \, \frac{\mu}{\sqrt{2 \pi x} \,(x^2+\mu)^{5/4} } \, .$$
Comparing this with the standard asymptotic
\begin{equation}
\label{stasymbes}
f(x) =\sqrt{x} \, J_\nu (x)=\sqrt{\frac{2}{\pi}} \, \cos{(x-\omega_\nu )}+
O(x^{-3/2}),
\end{equation}
for large $x$
one finds
$$ c_1=\sqrt{\frac{2}{\pi}} \, \sin  \omega_\nu , \; \; \; c_2= \sqrt{\frac{2}{\pi}} \, \cos \omega_\nu , $$
yielding (\ref{eqbes0}).

Similarly, for $\nu \ge \frac{1}{2} $ and $x \ge \mu, $ using (\ref{eqbess}) instead of (\ref{besszegj}), we obtain
$$g(x)=g_0 + \theta \, \frac{\mu}{\sqrt{8 \pi}} \int_x^\infty \frac{6z^2-\mu}{z \, (z^2-\mu)^{5/2}} \, dz =g_0+ \theta \, \frac{1}{\sqrt{8 \pi}} \, \left|\frac{3x^2+2 \mu}{3(x^2-\mu)^{3/2}} -\frac{\arcsin \frac{\sqrt{\mu}}{x}}{\sqrt{\mu}}\,\right|.$$
It is easy to check that the function
$$\frac{3x^2+2 \mu}{3(x^2-\mu)^{3/2}} -\frac{\arcsin \frac{\sqrt{\mu}}{x}}{\sqrt{\mu}} $$
is decreasing and positive. Therefore, by $\arcsin \frac{\sqrt{\mu}}{x} \ge  \frac{\sqrt{\mu}}{x}, \; x>0,$ we get
$$\frac{3x^2+2 \mu}{3(x^2-\mu)^{3/2}} -\frac{\arcsin \frac{\sqrt{\mu}}{x}}{\sqrt{\mu}} \le
\frac{3x^2+2 \mu}{3(x^2-\mu)^{3/2}} -\frac{1}{x} =$$
$$\frac{13 \mu}{6 (x^2-\mu)^{3/2}}\, -
\frac{(x+2 \sqrt{x^2-\mu} \,)(\sqrt{x^2-\mu}-x)^2}{2x (x^2-\mu)^{3/2}} <\frac{13 \mu}{6 (x^2-\mu)^{3/2}},$$
and
$$g(x)=g_0 + \theta \, \frac{13 \mu}{12 \sqrt{2 \pi}\, (x^2-\mu)^{3/2}}\, .$$
Comparing this with the asymptotic for large $x$ one finds
$$ c_1=\sqrt{\frac{2}{\pi}\,} \, \sin \omega_\nu, \; \; \; c_1=\sqrt{\frac{2}{\pi}\,} \, \cos \omega_\nu,$$
and (\ref{eqbes1}) follows.
\end{proof}

\begin{remark}
The approximation given by (\ref{eqbes1}) is strong enough to be matched with the solution in the transition region, as for $x=\nu+\nu^{1/3} z$ the error term is of order  $O(\nu^{-1/3} z^{-7/4}).$ This allows one to find the constants of integrations $c_1,c_2$ in Theorem \ref{thbessel} avoiding delicate claims like Lemma \ref{tailor}. We omit the details.
\end{remark}

Similarly to Corollary \ref{corairy} we obtain
\begin{corollary}
\begin{equation}
\label{airytoch}
Ai(-x)=
\end{equation}
$$ \frac{2 \sqrt{x} \cos \left(\frac{\sqrt{16x^3+5}}{6}- \frac{\sqrt{5}}{6} \,
\ln \frac{\sqrt{16x^3+5}+\sqrt{5}}{4x^{3/2}} -\frac{\pi}{4} \,\right)}{\sqrt{\pi} \, (16x^3+5)^{1/4}}
+\theta \frac{10 \sqrt{3}}{\sqrt{\pi} \, x^{1/4}(16x^3+5)^{3/2}} \,, $$
provided $ x >0.$
\end{corollary}

The argument of the cosine in (\ref{eqbes0}) and (\ref{airytoch}) (but not in (\ref{eqbes1})) can be simplified at the cost of a weaker numerical constant at the error term. Namely, it is easy to verify the following elementary inequalities:
$$\sqrt{x^2+\mu}+ \sqrt{\mu} \, \ln { \frac{x}{\sqrt{\mu} +\sqrt{\mu+x^2}} \,} =x-\frac{\mu}{2x}+\theta^2 \frac{\mu^2}{24 x^3} \, , \; \;x>0,$$

$$\frac{\sqrt{16x^3+5}}{6}- \frac{\sqrt{5}}{6} \,
\ln \frac{\sqrt{16x^3+5}+\sqrt{5}}{4x^{3/2}}= \frac{2}{3}\,x^{3/2}- \frac{5}{48}\,x^{-3/2}+\theta^2 \frac{25}{9216}\,x^{-9/2} \, .$$
Since $|\cos (x+\epsilon)-\cos x| \le \epsilon,$ we obtain
\begin{equation}
J_\nu(x)=
\sqrt{\frac{2}{\pi}} \, \frac{ \cos{ \left(x-\frac{\mu}{2x} - \omega_\nu \right) }}{(x^2+\mu)^{1/4}} + \theta \frac{25 \mu}{24 \sqrt{2 \pi } \,x^3 (x^2+\mu)^{1/4} } \, ,\;\; x>0,\; \;
|\nu| \le \frac{1}{2} \, ,
\end{equation}

\begin{equation}
Ai(-x)=\frac{2 \sqrt{x} \cos \left(\frac{2}{3}\,x^{3/2}- \frac{5}{48}\,x^{-3/2} -\frac{\pi}{4} \,\right)}{\sqrt{\pi} \, (16x^3+5)^{1/4}} + \theta \frac{5}{9 \sqrt{\pi} \, x^4 (16x^3+5)^{1/4}} \, .
\end{equation}

In particular, the last formula yields rather sharp approximations for the zeros of Airy (e.g. already for the first zero the error is less than 0.00122), and, in view of the inequality (\ref{firstzerbes}), the Bessel function.
\begin{theorem}
\begin{equation}
\label{airyzerosapprox}
a_s=16^{-2/3} \left( m+\sqrt{m^2+40} \, \right)^{2/3}+ \theta \frac{1280 \pi}{9m^3 (m^2+40)^{1/6}}\,,
\end{equation}
where $a_s$ is $s^{th}$ positive zero of $Ai(-x)$ and $m=(12s-3)\pi.$
\end{theorem}
\begin{proof}
Elementary arguments show that $|\sin x|< \epsilon$ implies $\sin (x+\frac{\pi \theta  \epsilon}{2}  ) =0$ for some
$\theta, \; \; |\theta| \le 1.$ Therefore, $Ai(-x)=0$ means
$$\frac{2}{3} \,x^{3/2}- \frac{5}{48} \,x^{-3/2} -\frac{3\pi}{4}+\theta \,
\frac{5\pi}{36 }\,x^{-9/2}=\pi s, \;\;\; s=0,1,2,...$$
Since $a_1 =2.33...,$ we may assume $x>2.$ Then
$$\frac{5}{48} \,x^{-3/2}> \frac{5\pi}{36 } \,x^{-9/2},$$
hence
$$ \frac{2}{3} \,x^{3/2}>\frac{3\pi}{4}+\pi s,$$
and we obtain the following estimate
$$\frac{5\pi}{36 }\, x^{-9/2} < \frac{640}{243\pi^2 (4s+3)^3} \, .$$
Thus, for $(s+1)^{th}$ zero this gives the equation
$$
\frac{2}{3} \,x^{3/2}- \frac{5}{48} \,x^{-3/2}=\frac{3\pi}{4}+\pi s+ \theta \,\frac{640 \pi }{9 m^3} :=\frac{3\pi}{4}+\pi s+ \epsilon ,
$$
with the relevant solution
$$x=16^{-2/3} \left(m+12 \epsilon+ \sqrt{(m+12 \epsilon)^2+40} \, \right)^{2/3}.$$
After some calculations one gets
$$x=16^{-2/3} \left(m+ \sqrt{m^2+40} \, \right)^{2/3}\left(
1+\theta \frac{8\epsilon}{\sqrt{m^2+40}} \, \right)=$$
$$
16^{-2/3} \left(m+ \sqrt{m^2+40} \, \right)^{2/3}+\theta \frac{1280 \pi}{9m^3 (m^2+40)^{1/6}}\,.
$$
This completes the proof.
\end{proof}
Formula (\ref{airyzerosapprox}) can be simplified at the cost of slightly weaker numerical constant. Namely, as one can checks
$$0 <\frac{1}{4} (m^2+20)^{1/3} - 16^{-2/3} \left( m+\sqrt{m^2+40} \, \right)^{2/3} <
\frac{25}{3 m^3 (m^2+40)^{1/6} } \, ,$$
 yielding
 \begin{equation*}
a_s=\frac{1}{4} (m^2+20)^{1/3}+ \theta \, \frac{456}{m^3 (m^2+40)^{1/6}} \, .
\end{equation*}
 Finally, comparing the numerical values of the zeros of $Ai(-x)$ with (\ref{airyzerosapprox}) leads to the following conjecture:
 \begin{equation*}
a_s<16^{-2/3} \left( m+\sqrt{m^2+40} \, \right)^{2/3}.
\end{equation*}


\end{document}